\AtBeginDocument{%
   \def\MR#1{}
}
\documentclass[11pt]{amsart}
\usepackage[margin=3cm]{geometry}
\usepackage{bm}
\usepackage{amsthm,amsmath,mathtools}
\usepackage{amssymb}
\usepackage{mathrsfs}
\usepackage{hyperref,xcolor}
\usepackage[foot]{amsaddr}
\usepackage{stmaryrd,esint}
\usepackage{graphicx}

\newcommand{\R}{\mathbb{R}}
\newcommand{\N}{\mathbb{N}}
\newcommand{\Z}{\mathbb{Z}}

\DeclareMathOperator{\area}{area}
\DeclareMathOperator{\vol}{vol}

\newtheorem{thm}{Theorem}
\newtheorem*{thm*}{Theorem}
\newtheorem*{namedthm}{\namedthmname}

\newtheorem*{lem*}{Lemma}
\newtheorem{prop}{Proposition}
\newtheorem{coro}{Corollary}
\newtheorem*{coro*}{Corollary}

\newtheorem*{problem}{Problem}

\theoremstyle{definition}
\newtheorem{defi}{Definition}

\newtheorem{rmk}{Remark}

\newcounter{namedthm}
\makeatletter
\newenvironment{named}[1]
  {\def\namedthmname{#1}%
   \refstepcounter{namedthm}%
   \namedthm\def\@currentlabel{#1}}
  {\endnamedthm}
\makeatother

\title{When can minimal hypersurfaces be connected by mean curvature flow?}
\author{Jingwen Chen\textsuperscript{1}, Pedro Gaspar\textsuperscript{2}}
\address{\parbox{\linewidth}{
\textsuperscript{1} Department of Mathematics, University of Pennsylvania, \\
David Rittenhouse Lab,
209 South 33rd Street,
Philadelphia, PA 19104 \\ \textsuperscript{2} Facultad de Matem\'aticas, Pontificia Universidad Cat\'olica de Chile \\ Avenida Vicuña Mackenna 4860, Santiago, Chile \smallskip}}
\email{jingwch@sas.upenn.edu, pedro.gaspar@mat.uc.cl}

\begin{document}

\begin{abstract}

From the perspective of Morse theory, it is natural to investigate gradient flow trajectories between critical points. In this short note, we explore the minimal hypersurface analogue of this phenomenon and present examples that suggest additional topological and variational obstructions to the existence of connecting mean curvature flows.

\end{abstract}

\maketitle

\section{introduction}

The area functional is one of the most widely studied functions in geometric analysis. Minimal surfaces arise as critical points of this functional, while mean curvature flow can be interpreted as its negative gradient flow. The development of \emph{min-max theory} was motivated by Morse-theoretic considerations in the space of hypersurfaces \cite{MarquesNevesIndex}. From the Morse theory perspective, a natural question to ask is: Can two critical points be connected by a gradient flow trajectory? This suggests the following minimal hypersurface analogue: 

\begin{problem}
Let $(M,g)$ be a Riemannian manifold. Given minimal hypersurfaces $\Sigma_-$, $\Sigma_+$ in $(M,g)$, under what conditions can we construct a (weak) eternal solution $\{\Sigma_t\}_{t \in \R}$ of the mean curvature flow in $(M,g)$ such that $\Sigma_t$ converges to $\Sigma_{\pm}$ as $t \to \pm \infty$, in a suitable sense?
\end{problem}

From the variational description of the mean curvature flow mentioned above, if $\Sigma_\pm$ are closed, then a natural necessary condition is that the areas of $\Sigma_-$ and $\Sigma_+$ satisfy $\area(\Sigma_-) \geq \area(\Sigma_+)$, since we expect $\area(\Sigma_t)$ to be nonincreasing. Based on considerations from Morse Homology and min-max theory for minimal hypersurfaces, we conjecture that if $(M,g)$ is closed, equipped with a bumpy metric or a metric with positive Ricci curvature, then for any integer $p \geq 2$, there exists a min-max minimal hypersurface $\Sigma_-$ with Morse index $p$, whose area equals the $p$-width of $(M,g)$, and which can be connected to a minimal hypersurface $\Sigma_+$ of lower area and lower Morse index.

In the previous work \cite{ChenGaspar}, the authors studied eternal Brakke flows — a geometric measure-theoretic weak mean curvature flow — connecting low area minimal surfaces of the round $3$-sphere, based on the resolution of Willmore conjecture by Marques-Neves \cite{MarquesNevesWillmore}. Later in \cite{chensymmetric, cgmorse}, we extend the results to higher-dimensional cases and in the Allen-Cahn setting, with further information about the space of connecting flows between low area symmetric minimal hypersurfaces in $S^3$ and $S^n$.\medskip

In this short note, building on the works of \cite{Whitetopology} and \cite{mramor2024long}, we provide examples that suggest other topological and variational obstructions to the existence of such connecting mean curvature flows. We consider the one-parameter family of \emph{Berger $3$-spheres} $S_{\tau}$, which are important examples of compact homogeneous $3$-manifolds; see Section \ref{berger} for a precise description of these metrics and further geometric properties. For any integer $p \geq 1$, by combining the explicit knowledge of the area of minimal spheres in $S_\tau$ \cite{Torralbo}, the upper bounds for the widths obtained by the authors recently in \cite{cgberger}, and the continuity of the widths in the space of Riemannian metrics \cite{MarquesNevesequidistribution}, we prove the existence of a Berger sphere $S_{\tau(p)}$ for which the area of the minimal spheres attains the $p$-width.

\begin{prop} \label{equator p}
For any integer $p \geq 1$, there exist $\tau(p) \in (0,1]$, such that
\begin{equation*}
\area_{S_{\tau(p)}} (S^2) = \omega_p(S_{\tau(p)}).    
\end{equation*}
\end{prop}

Together with the results of \cite{Whitetopology,mramor2024long}, this implies our main theorem:

\begin{named}{Theorem A} \label{nonexistence}
There exist infinitely many Riemannian spheres $(S^3,g)$ with metrics of positive Ricci curvature, each admitting a minimal surface $\Sigma^p$ whose area equals the $p$-width of $(S^3,g)$ for some integer $p \geq 1$ and satisfying the following property: for every minimal surface $\Sigma^q$ with smaller area, there is no eternal Brakke flow $\{\Sigma_t\}_{t \in (-\infty, \infty)}$ so that
\[\Sigma_t \to \Sigma^q \text{ as }t \to +\infty, \text{ and }\Sigma_t \to \Sigma^{p} \text{ as }t \to -\infty,\quad  \text{in the sense of varifold convergence}. \]
\end{named}\medskip

\begin{rmk}
This theorem does not provide a counterexample to the conjectured result stated below the Problem above concerning the existence of a connection from a min-max minimal hypersurface, since there may exist other minimal surfaces realizing the $p$-width that could be connected to a minimal surface with low index via an eternal Brakke flow. In fact, for generic (bumpy) metrics $g$ on $M$, it follows from \cite{MarquesNevesIndex,Zhou} that $\omega_p(M,g)$ is achieved by the area of an embedded minimal surface with Morse index $p$, whereas $S^2 \subset S_\tau$ has Morse index $1$ by \cite{TUIndex}. 

Using an approximation argument, together with the results of \cite{MarquesNevesIndex,Zhou} and compactness theorems for minimal surfaces with bounded index \cite{SharpCompactness}, we can find minimal surfaces $\Sigma^p$ in $S_\tau$ with area $\omega_p(S_\tau)$, Morse index $\leq p$, and whose index plus nullity (as a critical point of the area functional) is bounded below by $p$. It would be interesting to know whether these minimal surfaces can be connected via mean curvature flow to another minimal surface of lower index.
\end{rmk}

In addition, we provide a strengthened version of Proposition \ref{equator p}, with further information about the coincidence of the first $p$ widths of Berger spheres:

\begin{prop} \label{equator p not 1}
There exist $\tau \in (0,1)$ and $p > 1$, such that the min-max widths of the Berger $3$-sphere $S_\tau$ satisfy:
\begin{equation*}
\area_{S_{\tau}} (S^2) = \omega_p(S_{\tau}) > \omega_1(S_{\tau}).    
\end{equation*}
\end{prop}
By iterating the argument in the proof of this proposition, we obtain the following result:

\begin{coro} \label{infinite berger}
There exists an increasing sequence $\{q_i\}$ of positive integers and a corresponding decreasing sequence $\{\tau_{q_i}\}$ with $\tau_{q_i} \in (0,1]$ such that for each $i$, the min-max widths of the Berger $3$-spheres $S_{\tau_{q_i}}$ satisfy:
\begin{equation*}
\area_{S_{\tau_{q_i}}}(S^2) = \omega_{q_i}(S_{\tau_{q_i}}) > \omega_{q_i - 1}(S_{\tau_{q_i}}).
\end{equation*}
\end{coro}\bigskip

\subsection*{Acknowledgements}

We would like to thank André Neves and Ao Sun for many helpful discussions. PG was supported by ANID (Agencia Nacional de Investigaci\'on y Desarrollo) FONDECYT Iniciaci\'on grant.

\section{preliminary}

In this section, we present some background results needed for this paper.

\subsection{Min-max theory}

We briefly recall some basic definitions and important results about the volume spectrum of a closed Riemannian manifold $(M^n,g)$. We refer to \cite{MarquesNevesWillmore,MarquesNevesIndex,LMNWeyl} and the references therein for further details.

\begin{defi}
The (min-max) \emph{$p$-width} $\omega_p(M,g)$ is
\[\omega_p(M,g):= \inf\limits_{\Phi \in \mathcal{P}_p(M)} \sup\limits_{x \in X} \mathbf{M}(\Phi(x)),\]
where $\mathcal{P}_p(M)$ denotes the space of $p$-sweepouts of $M$ by $(n-1)$-dimensional mod $2$ cycles, and $\mathbf{M}$ stands for the \emph{mass} of such cycles.
\end{defi}

The non-decreasing sequence $\{\omega_p(M,g)\}$ is called the \emph{volume spectrum} of $(M,g)$. The asymptotic growth of this sequence was studied by L. Guth in \cite{guthminmax}. In addition, Liokumovich-Marques-Neves proved in \cite{LMNWeyl} that it satisfies the following Weyl-type law:

\begin{thm}[\cite{LMNWeyl}] \label{weyl law}
There exists a constant $a(n)>0$ such that for every compact Riemannian manifold $(M^{n},g)$ with (possibly empty) boundary, we have
\[\lim\limits_{p \to \infty} \omega_p(M) p^{-\frac{1}{n}} = a(n) \vol (M)^{\frac{n-1}{n}}.\]
\end{thm}

Moreover, we have the following main existence result in the min-max theory.

\begin{thm*}[{\cite{MarquesNevesIndex}, \cite[Theorem C]{Zhou}}] \label{widths are achieved} For every $p \in \N$, there exists a stationary integral varifold $V$ with $\mathrm{spt}\|V\| = \Sigma$ such that
    \begin{enumerate}
        \item[(i)] $\omega_p(M,g) = \|V\|(M)$,
        \item[(ii)] $\Sigma$ is a minimal hypersurface with \emph{optimal regularity}, namely it is smooth away from a singular set of Hausdorff dimension at most $(n-8)$,
        \item[(iii)] The Morse index of $\Sigma$ satisfies $\mathrm{ind}(\Sigma) \leq p$.
    \end{enumerate}
\end{thm*}

\subsection{Mean curvature flow} \label{mcf}

A smooth one-parameter family of embedded hypersurfaces $\Sigma(t) \subset (M,g),\ t \in (0, T)$ evolves by its mean curvature if 
\[
\partial_t x=\vec H(x),
\]
where $\vec H(x)$ is the mean curvature vector of the hypersurface at point $x$.

Singularities are generally unavoidable in mean curvature flow, which forces us to study weak solutions of mean curvature flow. There are (at least) three standard methods to continue the flow past singularities: Brakke flow  \cite{Brakke}, mean curvature flows with surgery \cite{HuiskenSinestrari}, and level set flow \cite{EvansSpruck}. These generalized solutions are called \emph{weak} mean curvature flows.

In \cite{Whitetopology}, White analyzed the topological behavior of the complement of the level set flow. Roughly speaking, White proved that the topology under evolution could only become simpler. Combining this work with the Brakke regularity theorem, the results from Brendle \cite{Brendleshrinker}, Choi–Haslhofer–
Hershkovits–White \cite{chhwancient}, and Bamler-Kleiner \cite{bamler2023multiplicity}, we obtain the following result:

\begin{prop} \label{genus non increasing}

Let $(M^3, g)$ be a $3$-dimensional closed Riemannian manifold. Let $\{ \Sigma_t \}_{t \in (-\infty, \infty)}$ be an eternal Brakke flow on $M$. Suppose that $\Sigma_t$ converges (in the sense of varifold convergence) to smooth minimal surfaces $\Sigma_{\pm \infty}$ as $t \to \pm \infty$ respectively, where $\Sigma_{+ \infty}$ may have multiplicity. If $\Sigma_{-\infty}$ is a sphere of multiplicity $1$, then $\Sigma_{+\infty}$ has genus $0$.

\end{prop}

See, e.g., \cite[Proposition 3.7]{mramor2024long} for a proof. The flow we considered is of genus $0$ and thus will not fatten; the innermost and outermost flows coincide. For more discussion about the long-time limit of mean curvature flow in closed manifolds, see \cite{Ilmanen95_Sing2D, mramor2024long}.

\subsection{Berger sphere} \label{berger}

Recall the one-parameter family of \emph{Berger metrics} on $S^3 = \{z=(z_1,z_2) \in \mathbb{C}^2 \colon |z_1|^2+|z_2|^2=1\}$:
\[
\left\langle v,w \right\rangle_{\tau} = \left\langle u,v \right\rangle - (1-\tau^2) \left\langle v,iz\right\rangle \left\langle w, iz\right\rangle, \qquad v,w \in T_zS^2,
\]
where $\left\langle \cdot, \cdot \right\rangle$ is the standard Euclidean metric in $\mathbb{C}^2$. For every $\tau \in \R_{>0}$, this defines a homogeneous metric on $S^3$ with $4$-dimensional isometry group, with $\left\langle\cdot,\cdot\right\rangle_1=\left\langle\cdot,\cdot\right\rangle=$ round metric on $S^3$. The Berger metrics have positive Ricci curvature (see e.g. \cite{TUIndex}). Let $S_{\tau}$ denote $S^3$ with the Berger metric $\left\langle\cdot,\cdot\right\rangle_\tau$, thus $S_1$ is the round $3$-sphere. The standard equatorial sphere $S^2=\{(z_1,z_2) \in S^3 \colon\ \mathrm{Im}z_2 =0\}$ -- and hence any other equatorial sphere, as $S^3_\tau$ is homogeneous -- and the Clifford torus $T=S^1(1/\sqrt{2})\times S^1(1/\sqrt{2})$ are minimal surfaces in $S_{\tau}$ for every $\tau$, and their area were explicitly computed by Torralbo \cite[Proposition 2]{Torralbo}. 

It follows from \cite{torralbo2009compact} that the equatorial spheres are, in fact, the only embedded minimal spheres in $S_\tau$. In addition, their area $A(\tau):=\area_{S_\tau}(S^2)$ is explicitly computed in \cite[Proposition 2]{Torralbo}, and satisfy $A(\tau) \geq 2\pi$.

In the previous work \cite{cgberger}, the authors study the volume spectrum of Riemannian bundles. In particular, for the Berger spheres, we used the computation of the volume spectrum of the $2$-sphere by \cite{CMSurfaces} to prove the following upper bound for the min-max widths:
\begin{equation} \label{berger width upper bound}
\omega_p(S_\tau) \leq 2\pi^2\tau\lfloor \sqrt{p} \rfloor, \quad \text{for every} \ p \in \N.
\end{equation}

\section{proof of the main theorem}

\begin{proof}[Proof of Proposition \ref{equator p}]

Note that in the round $3$-sphere $S^3=S_1$, the area of the equatorial $2$-spheres realizes the $p$-widths for $p = 1,2,3,4$. Therefore, we may set $\tau(p) = 1$ for $p = 1,2,3,4$ and restrict our consideration to the case $p \geq 5$.

For $p \geq 5$, by \eqref{berger width upper bound}, when $\tau \leq \frac{1}{\pi \lfloor \sqrt{p} \rfloor}$,
\begin{equation*}
\omega_p(S_\tau) \leq 2\pi^2\tau\lfloor \sqrt{p} \rfloor \leq 2\pi \leq A(\tau).
\end{equation*}

On the other hand, for the round $3$-sphere $S_1$, it follows from the solution of the Willmore conjecture by Marques-Neves \cite{MarquesNevesWillmore} that:
\begin{equation*}
\omega_p (S_1) \geq \omega_5 (S_1) = 2\pi^2 > 4\pi = A(1).
\end{equation*}
By the Intermediate value theorem and the continuity of the functions $A(\tau)$ and $\omega_p(S_{\tau})$ (observe that the metrics $\langle\cdot,\cdot\rangle_\tau$ in $S^3$ depend smoothly on $\tau$, and $\omega_p$ depends continuously on the metric, see e.g. Lemma $1$ of \cite{MarquesNevesequidistribution}), there exists $\tau(p) \in [\frac{1}{\pi \lfloor \sqrt{p} \rfloor},1)$ such that 
$$\omega_p (S_{\tau(p)})= A({\tau(p)}).$$\qedhere
\end{proof}

We are now ready to prove our main theorem. 

\begin{proof}[Proof of \ref{nonexistence}]

For any $p \geq 1$, by Proposition \ref{equator p}, the equatorial spheres attain the $p$-width of $S_{\tau(p)}$. For every minimal surface $\Sigma^q$ with area $< A(\tau(p))$, since the equatorial spheres are the only minimal spheres in $S_{\tau(p))}$, $\Sigma^q$ must have genus $\geq 1$. By Proposition \ref{genus non increasing}, there is no eternal Brakke flow connecting an equatorial sphere and $\Sigma^q$.
\end{proof}

\begin{rmk}
Since the equatorial sphere $S^2 \subset S_\tau$ has Morse index 1 and $S_\tau$ contains no stable minimal surfaces, we can combine the results of Choi-Mantoulidis \cite{ChoiMantoulidis} and White \cite{White-meanconvex} (together with Brakke regularity for mean curvature flow) to conclude that any ancient Brakke flow $\{\Sigma_t\}_{t \in (-\infty,T]}$ with $\Sigma_t \to S^2$ as $t \to -\infty$ in the varifold sense must actually disappear in finite time.
\end{rmk}

\section{two proofs of Proposition \ref{equator p not 1}}

We provide two proofs of Proposition \ref{equator p not 1}, which gives us some interesting byproducts concerning the volume spectrum of the Berger spheres.

\begin{proof}[First Proof of Proposition \ref{equator p not 1}]

Let $f: (0,1] \to \Z_{\geq 0}$ be the function defined by
\[
f(\tau) := 
\begin{cases} 
\max\bigl\{k \in \mathbb{N} \mid \omega_1(S_\tau) = \cdots = \omega_k(S_\tau) = A(\tau)\bigr\} & \text{if } \omega_1(S_\tau) = A(\tau), \\ 
0 & \text{otherwise.}
\end{cases}
\]

We know that $f(1) = 4$, and that $f(\tau) = 0$ for $\tau < \frac{1}{\pi}$ by Corollary $D$ in \cite{cgberger}.\medskip

\noindent\textbf{Claim:} $f$ is bounded. 

\noindent \emph{Proof of Claim.} Otherwise, there exist a sequence $\{\tau_j\} \subset (0,1]$ such that $f(\tau_j)>j$ for every $j \in \N$. By the observation above, we may assume $\{\tau_j\} \subset [\frac{1}{\pi},1]$ and, passing to a subsequence if necessary, that $\tau_i \to \tau$ for some $\tau \in [\frac{1}{\pi}, 1]$.

Since $m \geq j$ implies $f(\tau_m) >j$, by the definition of $f$,
\[\omega_1(S_{\tau_m}) =\cdots = \omega_j(S_{\tau_m}) = A(\tau_m), \text{ for all } m\geq j.\]
If we let $m \to + \infty$, then by the continuity of widths with respect to the Riemannian metric, by the continuity of $A(\tau)$ and by $\tau_m \to \tau$, we obtain 
\[\omega_1(S_{\tau}) =\cdots = \omega_j(S_{\tau}) = A(\tau).\]
Since $j$ is arbitrary, we conclude $\omega_j(S_{\tau}) = A(\tau)$ for all $j$. This contradicts Theorem \ref{weyl law}. \hfill\qed
\smallskip

Hence, there exists $N \geq 4$ such that $f(\tau) \leq N$ for all $\tau \in (0,1]$. The same argument as in the proof of Proposition \ref{equator p} implies the existence of $\tau' \in [\frac{1}{\pi \lfloor \sqrt{N+1} \rfloor},1)$ satisfying 
$$\omega_{N+1} (S_{\tau'})= A({\tau'}).$$ 

By the definitions of $N$ and $f$, we know that $\omega_1(S_{\tau'}) \neq A(\tau')$. Therefore 
$$A(\tau')=\mathrm{area}_{S_{\tau'}}(S^2) = \omega_{N+1}(S_{\tau'}) > \omega_1(S_{\tau'}).$$
\end{proof}

We now present a proof that achieves a similar conclusion without directly invoking the function 
$f$ or relying on contradiction arguments.

\begin{proof}[Second Proof of Proposition \ref{equator p not 1}]

Let $\tau' = \inf\limits_{\tau \in (0,1]} \{\omega_1(S_{\tau}) = A(\tau)\}$.

By Corollary D in \cite{cgberger}, we know that $\tau' \in [\frac{1}{\pi}, 1]$, in particular, $\tau' > 0$. By Theorem \ref{weyl law}, there exists a positive integer $p_1 \geq 2$ such that
\[\omega_{p} (S_{\tau'}) > A(\tau'), \text{ for any } p \geq p_1. \]

For any such $p$, the same argument as in the proof of Proposition \ref{equator p} implies the existence of $\tau_p \in [\frac{1}{\pi \lfloor \sqrt{p} \rfloor},\tau')$ such that
\[\omega_p (S_{\tau_p}) = A(\tau_p).\]

Since $\tau_p < \tau'$, we have  $\omega_1 (S_{\tau_p}) \neq A(\tau_p)$. Therefore
\[A(\tau_p) = \omega_p(S_{\tau_p}) > \omega_1 (S_{\tau_p}).\qedhere\]
\end{proof}

The proof of Corollary \ref{infinite berger} is based on iterating the above method by replacing the $1$-width with the $p_i$-width.

\begin{proof}[Proof of Corollary \ref{infinite berger}]

Let $p_0 = 1$. Let $p_1, \tau_{p_1}$ be as in the second proof, which implies
\[A(\tau_{p_1}) =  \omega_{p_1} (S_{\tau_{p_1}}) > \omega_1 (S_{\tau_{p_1}}).\]

Suppose $p_i$ and $\tau_{p_i}$ are already defined with $A(\tau_{p_i})= \omega_{p_i} (S_{\tau_{p_i}}) > \omega_{p_{i-1}} (S_{\tau_{p_i}})$. For any $\tau$ such that $\omega_{p_i}(S_{\tau}) = A(\tau)$, by \eqref{berger width upper bound},
\[2\pi \leq A(\tau) = \omega_{p_i}(S_{\tau}) \leq 2\pi^2 \tau \lfloor \sqrt{p_i} \rfloor.\]
Thus $\tau \geq \frac{1}{\pi \lfloor \sqrt{p_i} \rfloor}.$ Let $\tau' = \inf\limits_{\tau \in (0,1]} \{\omega_{p_i}(S_{\tau}) = A(\tau)\}$. Then $0 < \frac{1}{\pi \lfloor \sqrt{p_i} \rfloor} \leq \tau' \leq \tau_{p_i}$.

By Theorem \ref{weyl law}, there exists a positive integer $p_{i+1} > p_i$ such that
\[\omega_{p_{i+1}} (S_{\tau'}) > A(\tau'). \]

Again, the same argument as in the proof of Proposition \ref{equator p} implies the existence of $\tau_{p_{i+1}} \in [\frac{1}{\pi \lfloor \sqrt{p_{i+1}} \rfloor}, \tau')$ satisfying 
$$\omega_{p_{i+1}} (S_{\tau_{p_{i+1}}}) = A(\tau_{p_{i+1}}).$$

Since $\tau_{p_{i+1}} < \tau'$, we have  $\omega_{p_i} (S_{\tau_{p_{i+1}}}) \neq A(\tau_{p_{i+1}})$. Therefore
\[A(\tau_{p_{i+1}}) = \omega_{p_{i+1}}(S_{\tau_{p_{i+1}}}) > \omega_{p_i} (S_{\tau_{p_{i+1}}}).\]

This induction argument gives us an increasing sequence $\{p_i\}$ and a decreasing sequence $\{\tau_{p_i}\}$. Let $q_i$ be the smallest integer such that $\omega_{q_i} (S_{\tau_{p_i}}) = A(\tau_{p_i})$, and set $\tau_{q_i} = \tau_{p_i}$. Then
\[A(\tau_{q_i}) = \omega_{q_i}(S_{\tau_{q_i}}) > \omega_{q_i - 1}(S_{\tau_{q_i}}).\]

Since $\omega_{q_i}(S_{\tau_{q_i}}) = A(\tau_{q_i}) = A(\tau_{p_i}) > \omega_{p_{i-1}} (S_{\tau_{p_i}}) = \omega_{p_{i-1}} (S_{\tau_{q_i}})$, we have $p_{i-1} < q_i \leq p_i$, which implies that $\{q_i\}$ is an increasing sequence.
\end{proof}

\begin{rmk}
We observe that in the proofs above, we need not use the full strength of Theorem \ref{weyl law}. We rely simply on the unboundedness of $\omega_p(S_\tau)$, which follows, for instance, from \cite{guthminmax}.
\end{rmk}

\bibliographystyle{amsalpha}
\bibliography{main}

\end{document}